\documentclass[10pt]{amsart}
\usepackage{amssymb}
\usepackage[all]{xy}

\DeclareMathOperator{\cd}{cd}

\DeclareMathOperator{\Inf}{inf}
\DeclareMathOperator{\Ker}{Ker}
\DeclareMathOperator{\Res}{res}

\DeclareMathOperator{\rk}{rk}

  \newtheorem{thm}{Theorem}[section]
  \newtheorem{pro}[thm]{Proposition}
  \newtheorem{cor}[thm]{Corollary}
  \newtheorem{lem}[thm]{Lemma}
  \newtheorem{conj}[thm]{Conjecture}

  \theoremstyle{definition}
  \newtheorem{rem}[thm]{Remark}
  \newtheorem{defn}[thm]{Definition}
  \newtheorem{exam}[thm]{Example}

  \numberwithin{equation}{section}

\newcommand{\F}{\mathbb{F}}
\newcommand{\dbN}{\mathbb{N}}

\newcommand{\Z}{\mathbb{Z}}
\newcommand{\calG}{\mathcal{G}}
\newcommand{\K}{\mathbb{K}}
\newcommand{\FpX}{\mathbb{F}_p\langle X\rangle}
\newcommand{\calL}{\mathcal{L}}
\newcommand{\calV}{\mathcal{V}}
\newcommand{\calE}{\mathcal{E}}
\newcommand{\calB}{\mathcal{B}}
\newcommand{\calX}{\mathcal{X}}

\begin{document}

\title[Pro-{$p$} groups and universal Koszulity]{Pro-$p$ groups with few relations \\ and universal Koszulity}

\author{Claudio Quadrelli}
\address{Department of Mathematics\\
University of Milano-Bicocca\\
Via R. Cozzi 55 -- Ed. U5 \\
20125 Milan\\
Italy, EU}
\email{claudio.quadrelli@unimib.it}

\begin{abstract}
 Let $p$ be a prime. We show that if a pro-$p$ group with at most 2 defining relations has
quadratic $\F_p$-cohomology algebra, then this algebra is universally Koszul. This proves the
``Universal Koszulity Conjecture'' formulated by J.~Min\'a\v{c} et al. in the case of maximal pro-$p$
Galois groups of fields with at most 2 defining relations.
 \end{abstract}

\maketitle

\section{Introduction}
\label{Intro}

Let $k$ be a field, and $A_\bullet=\bigoplus_{n\geq0} A_n$ a $k$-algebra graded by $\dbN$.
The algebra $A_\bullet$ is quadratic if it is {\sl 1-generated} --- i.e., every element is a combination of products
of elements of $A_1$ ---, and its relations are generated by homogeneous relations of degree 2. 
E.g., symmetric algebras and exterior algebras are quadratic.

A quadratic algebra is called {\sl Koszul algebra} if $k$ admits a resoultion of free $\dbN$-graded right $A_\bullet$-modules such that for every $n\geq0$ the subspace of degree $n$ of the $n$-th term of the resolution is finitely generated, and such subspace generates the respective module (see sec.~\ref{ssec:UK}).
Koszul algebras were introduced by S. Priddy in \cite{priddy}, and they have exceptionally nice
behavior in terms of cohomology (see \cite[Ch.~2]{poliposi}). Koszul property is very restrictive, still it arises in various areas of mathematics, such as representation theory, algebraic geometry, combinatorics. Hence, Koszul algebras have become an important ob{\bf j}ect of study.

Recently, some stronger versions of the Koszul property were introduced and investigated (see, e.g., \cite{con:UK,CTV,piont,con:K,MPPT,pal}).
For example, the universal Koszul property (see Definition~\ref{defn:UK} below), which implies ``simple'' Koszulity.
Usually, checking whether a given quadratic algebra is Koszul is a rather hard problem.
Surprisingly, testing universal Koszulity may be easier, even though it is a more
restrictive property. 

Quadratic algebras and Koszul algebras have a prominent role in Galois theory, too. 
Given a field $\K$, for $p$ a prime number let $\calG_{\K}$ denote the maximal pro-$p$ Galois group of $\K$ --- namely, 
$\calG_{\K}$ is the Galois group of the maximal $p$-extension of $\K$.
If $\K$ contains a root of 1 of order $p$ (and also $\sqrt{-1}$ if $p = 2$), then the celebrated Rost-Voevodsky
Theorem (cf. \cite{voev,weibel}) implies that the $\F_p$-cohomology algebra 
$H^\bullet(\calG_{\K},\F_p) =\bigoplus_{n\geq0} H^n(\calG_{\K},\F_p)$
of the maximal pro-p Galois group of $\K$, endowed with the graded-commutative cup product
\[H^s(\calG_{\K},\F_p) \times H^t(\calG_{\K},\F_p)\overset{\cup}{\longrightarrow} H^{s+t}(\calG_{\K},\F_p), \qquad s, t \geq  0,\]
is a quadratic $\F_p$-algebra. 
Koszul algebras were studied in the context of Galois theory by L. Positselski and A. Vishik (cf. \cite{posivisi,posi:K}, see also \cite{MPQT}), and Positselski conjectured that the algebra $H^\bullet(\calG_{\K},\F_p)$ is Koszul (cf. \cite{posi:number}).
Positselski's conjecture was strengthened by J. Min\'a\v{c} et al. (cf. \cite[Conj. 2]{MPPT}):

\begin{conj}\label{conjecture:intro}
Let $\K$ be a field containing a root of 1 of order p, and suppose that the quotient $\K^\times/(\K^\times)^p$ is finite.
Then the $\F_p$-cohomology algebra $H^\bullet(\calG_{\K},\F_p)$ of the maximal pro-$p$ Galois group of $\K$ is universally Koszul.
\end{conj}

Here $\K^\times$ denotes the multiplicative group $\K\smallsetminus\{0\}$, and by Kummer theory
$\K^\times /(\K^\times)^p$ is finite if and only if $\calG_{\K}$ is a finitely generated pro-$p$ group.

In this paper we study universal Koszulity for the $\F_p$-cohomology algebra of pro-$p$ groups with at most two {\sl defining relations}.  
A pro-$p$ group G has $m$ defining relations, with $m\geq0$, if there is a minimal pro-$p$ presentation $F/R$ of $G$ 
--- i.e., $G\simeq F/R$ with $F$ a free pro-$p$ group and $R$ a closed normal subgroup contained in the Frattini subgroup of $F$ --- such that $m$ is the minimal number of generators of $R$ as closed normal subgroup of $F$.
We prove the following.

\begin{thm}\label{thm:intro}
 Let $G$ be a finitely generated pro-$p$ group with {\bf a}t most two defining relations.
 If the $\F_p$-cohomology algebra $H^\bullet(G,\F_p)$ is quadratic (and moreover if $a^2 = 0$ for every $a \in H^\bullet(G,\F_2)$,
 if $p = 2$), then $H^\bullet(G,\F_p)$ is universally Koszul.
\end{thm}

The above result settles positively Conjecture~\ref{conjecture:intro} for fields whose maximal
pro-$p$ Galois group has at most two defi{\bf n}ing relations.

\begin{cor}\label{cor:intro}
 Let $\K$ be a field containing a root of 1 of order $p$ (and also $\sqrt{-1}$ if $p = 2$),
 and suppose that the quotient $\K^\times /(\K^\times)^p$ is finite.
 If $\calG_{\K}$ has at most two defining relations then the $\F_p$-cohomology algebra $H^\bullet(\calG_{\K},\F_p)$ is
universally Koszul.
\end{cor}

Note that the condition on the number of defining relations of $\calG_{\K}$ {\bf m}ay be formulated both in terms of the dimension of $H^2(\calG_{\K},\F_p)$ and in terms of the Brauer group of $\K$:
namely, $\calG_{\K}$ has at most two defining relations if and only if $\dim(H^2(\calG_{\K},\F_p))\leq2$, and if and only if the $p$-part of the Brauer group of $\K$ has rank at most two.

Theorem~\ref{thm:intro} can not be extended to pro-$p$ groups with quadratic $\F_p$-cohomology
with more than two defining relations, as there are finitely generated pro-$p$ groups
with three defining relations whose $\F_p$-cohomology algebra is quadratic and Koszul, but not
universally Koszul (see Example~\ref{exam:square}). 
Still, such examples are expected not to contradict Conjecture~\ref{conjecture:intro}, as these pro-$p$ groups are conjectured not to occur as maximal pro-$p$ Galois of groups of fields (see Remark~\ref{rem:PZ}).


\section{Quadratic algebras and Koszul algebras}\label{sec:quadalgebras}

Throughout the paper every graded algebra $A_\bullet=\bigoplus_{n\in\Z}A_n$ is tac{\bf i}tly assumed to be a unitary associative algebra over the finite field $\F_p$, and non-negatively graded of finite-type, i.e., $A_0=\F_p$, $A_n = 0$ for $n < 0$ and $\dim(A_n)<\infty$ for $n\geq1$.
For a complete account on graded algebras and their cohomology, we direct the reader to the first chapters of \cite{poliposi} and of \cite{ldval}, and to \cite[\S~2]{MPQT}.


\subsection{Quadratic algebras}\label{ssec:quad}

A graded algebra $A_\bullet=\bigoplus_{n\geq0}A_n$ is said to be {\sl graded-commutative} if o{\bf n}e has
\begin{equation}\label{eq:gradcomm}
 b\cdot a = (-1)^{ij} a\cdot b \qquad\text{for every }a\in A_i , b\in A_j.
\end{equation}
In particular, if $p$ is odd then one has $a^2 = 0$ for all $a\in A_\bullet$, whereas if $p = 2$ then a graded-commutative algebra is commutative. 
Furthermore, if $p = 2$ we call a commutative algebra $A_\bullet$ which satisfies $a^2 = 0$ for all $a\in A_\bullet$ a {\sl wedge-commutative} $\F_2$-algebra.

For a graded ideal $I$ of a graded algebra $A_\bullet$, $I_n$ denotes the intersection $I\cap A_n$ for every $n\geq0$, i.e., $I=\bigoplus_{n\geq0} I_n$.
For a subset $\Omega\subseteq A_\bullet$, $(\Omega)\unlhd A_\bullet$ denotes the two-sided graded ideal generated by $\Omega$.
Also, $A_+$ denotes the {\sl augmentation ideal} of $A_\bullet$, i.e., $A_+=\bigoplus_{n\geq1} A_n$.
Henceforth all ideals are assumed to be graded.

Given a finite vector space $V$, let $T_\bullet(V)=\bigoplus_{n\geq0}V^{\otimes n}$ denote the tensor algebra generated by $V$. The product of $T_\bullet(V)$ is induced by the tensor product, i.e., $ab = a\otimes b\in V^{\otimes s+t}$ for $a\in V^{\otimes s}$ , $b \in V^{\otimes t}$.

\begin{defn}\label{defn:quadratic}
 A graded algebra $A^\bullet$ is said to be quadratic if one has {\bf a}n isomorphism
\[ A_\bullet\simeq T_\bullet(A_1)(\Omega)\]
for some subset $\Omega\subseteq A_1\otimes A_1$. In this case we write $A_\bullet= Q(V, \Omega)$.
\end{defn}

\begin{exam}\label{exam:quad}
 let $V$ be a finite vector space.
 \begin{itemize}
  \item[(a)] The tensor algebra $T_\bullet(V)$ and the trivial quadratic algebra $Q(V, V^{\otimes2})$ are quadratic algebras.
 \item[(b)] The symmetric algebra $S_\bullet(V)$ and the exterior algebra $\Lambda_\bullet(V)$ are quadratic, as one has 
 $S_\bullet(V)=Q(V,\Omega_S)$ and $\Lambda_\bullet(V)=Q(V,\Omega_\wedge)$ with
\[\Omega_S = \{u \otimes v-v\otimes u \mid u, v\in V \}
\qquad\text{and}\qquad
\Omega_\wedge=\{u\otimes v+v\otimes u \mid u,v\in V \}.\]
 \item[(c)] Let $\F_p\langle X\rangle$ be the free algebra generated by the indeterminates $X =\{X_1,\ldots,X_d\}$.
 Then $\FpX$ is a graded algebra, with the grading induced by the subspa{\bf c}es of homogeneous polynomials.
 If $\Omega=\{f_1,\ldots,f_m\}\subseteq\FpX$ is a set of homogeneous polynomials of degree 2, then $\FpX/(\Omega)$ is a quadratic algebra.
 \end{itemize}
\end{exam}

\begin{exam}\label{exam:prod}
Let $A_\bullet=Q(A_1,\Omega_A)$ and $B_\bullet=Q(B_1,\Omega_B)$ be two quadratic algebras.
The direct product of $A_\bullet$ and $B_\bullet$ is the quadratic algebra $$A_\bullet\sqcap B_\bullet= Q(A_1\oplus B_1,\Omega),$$
with $\Omega=\Omega_A\cup \Omega_B\cup (A_1\otimes B_1)\cup(B_1\otimes A_1)$.
\end{exam}

\begin{exam}\label{exam:graph}
 Let $\Gamma=(\mathcal{V},\mathcal{E})$ be a finite combinatorial graph (without loops) --- 
 namely $\mathcal{V}=\{v_1 ,\ldots, v_d\}$ is the set of vertices of $\Gamma$, and 
 $$\mathcal{E} \subseteq\{\{v, w\} \mid v, w \in \mathcal{V}, v\neq w\} = \mathcal{P}_2(\mathcal{V})\smallsetminus\Delta(\mathcal{V})$$
is the set of edges of $\Gamma$ ---, and let V be the space with basis $\mathcal{V}$.
The {\sl exterior Stanley-Reisner algebra} $\Lambda_\bullet(\Gamma)$ associated to $\Gamma$ is the quadratic algebra
\[
 \Lambda_\bullet(\Gamma)=\dfrac{\Lambda_\bullet(V)}{(v\wedge w\mid \{v,w\}\notin \mathcal{E})}.
\]
In particular, $\Lambda_\bullet(\Gamma)$ is graded-commutative (wedge-commutative if $p=2$), and if $\Gamma$ is complete (i.e., $\mathcal{E}=\mathcal{P}_2(\mathcal{V})\smallsetminus\Delta(V))$ then $\Lambda_\bullet(\Gamma)\simeq\Lambda_\bullet(V)$, whereas if $\mathcal{E}=\varnothing$, then $\Lambda_\bullet(\Gamma)\simeq Q(V,V^{\otimes2})$.
\end{exam}


\subsection{Koszul algebras and universally Koszul algebras}\label{ssec:UK}

A quadratic algebra $A_\bullet$ is said to be {\sl Koszul} if it admits a resolution
\[
 \xymatrix{ \cdots\ar[r] & P(2)_\bullet\ar[r] & P(1)_\bullet\ar[r] & P(0)_\bullet \ar[r] & \F_p }
\]

of right $A_\bullet$-modules, where for each $i\in\dbN$, $P(i)_\bullet=\bigoplus_{n\geq0}P(i)_n$ is a free {\sl graded}
$A_\bullet$-module such that $P(n)_n$ is finitely generated for all $n\geq0$, and $P(n)_n$ generates $P(n)_\bullet$ as graded $A_\bullet$-module (cf. \cite[Def. 2.1.1]{poliposi} and \cite[\S~2.2]{MPQT}).
Koszul algebras have an exceptionally nice {\bf b}ehavior in terms of cohomology.
Indeed, if a quadratic algebra $A_\bullet=Q(V,\Omega)$ is Koszul, then one has an isomorphism of quadratic algebras
\begin{equation}\label{eq:quaddual}
  \bigoplus_{n\geq0}\mathrm{Ext}_{A_\bullet}^{n,n}(\F_p,\F_p)\simeq Q(V^\ast,\Omega^\perp),
\end{equation}
where $V^\ast$ denotes the $\F_p$-dual of $V$, and $\Omega^\perp\subseteq(V\otimes V)^{\ast}$ is the orthogonal of $\Omega\subseteq V\otimes V$ (cf. \cite{priddy}) --- since $V$ is finite, we identify $(V^\ast)^{\otimes2}=(V^{\otimes2})^\ast$ ---,
whereas $\mathrm{Ext}_{A_\bullet}^{i,j}(\F_p,\F_p)=0$ for $i\neq j$ (in fact, this is an equivalent definition of
the Koszul property).

On the one hand, by \eqref{eq:quaddual} it is very easy to compute the $\F_p$-cohomology of a Koszul algebra.
On the other hand, in general it is extremely hard to establish whether a given quadratic algebra is Koszul.
For this reason, some ``enhanced forms'' of Kosz{\bf u}lity --- which are stronger than ``simple'' Koszulity, but at the same time easier to check --- have been introduced by several authors. 
We give now the definition of {\sl universal Koszulity} as introduced in \cite{MPPT}.

Given two ideals $I, J$ of a graded algebra $A_\bullet$, the {\sl colon ideal} $I : J$ is the ideal
\[ I : J = \{a\in A_\bullet\mid a\cdot J \subseteq I\}.\]

\begin{rem}\label{rem:colon}
 Note that for every two ideals $I, J$ of $A_\bullet$, the colon ideal $I : J$ contains all $a\in A_\bullet$ such that 
 $a\cdot J = 0$, as $0 \in I$.
 Moreover, if $A_\bullet$ is graded-commutative (and wedge-commutative, if $p = 2$), then {\bf f}or every for every $b\in J$
one has $b\in I : J$, as $b\cdot b = 0$.
\end{rem}

For a quadratic algebra $A_\bullet$, let $\calL(A_\bullet)$ denote the set of all ideals of $A_\bullet$ generated by a subset of $A_1$, namely, 
$$\calL(A_\bullet) = \{I \in A_\bullet \mid I=A_\bullet\cdot I_1 \}.$$
Note that both the trivial ideal $(0)$ and the augmentation ideal $A_+$ belong to $\calL(A_\bullet)$.

\begin{defn}\label{defn:UK}
 A quadratic algebra $A_\bullet$ is said to be {\sl universally Koszul} if for every ideal $I\in \calL(A_\bullet)$, and every $b\in A_1\smallsetminus I_1$, one has $I:(b)\in\calL(A_\bullet)$.
\end{defn}

Universal Koszulity is stronger than Koszulity, since every quadratic algebra
which is universally Koszul is also Koszul (cf. \cite[\S~2.2]{MPPT}).

\begin{exam}\label{exam:UK}
 \begin{itemize}
  \item[(a)] Let $V$ be a vector space of {\bf f}inite dimension. Then both the trivial algebra $Q(V,V^{\otimes2})$ (by definition) and the exterior algebra $\Lambda_\bullet(V)$ (by \cite[Prop. 31]{MPPT}) are universally Koszul. 
\item[(b)] If $A_\bullet$ and $B_\bullet$ are two quadratic universally Koszul algebras, then also
the direct product $A_\bullet\sqcup B_\bullet$ is universally Koszul (cf. \cite[Prop. 30]{MPPT}).
 \end{itemize}
\end{exam}

\begin{exam}\label{exam:DemushkinUK}
 For $V$ a finite vector space of even dimension $d$ and basis $\{v_1,\ldots,v_d\}\subseteq V$, let $A_\bullet$ be the quadratic algebra $A_\bullet=\Lambda_\bullet(V)/(\Omega)$, where
\[
\Omega=\left\{\begin{array}{c} v_1\wedge v_2-v_i\wedge v_{i+1}, \text{ for }i =1,3,\ldots,d-1, \\
v_i\wedge v_j,\text{ for }i<j,(i,j)\neq(1,2),(3,4),\ldots,(d-1,d)
              \end{array}\right\}\subseteq \Lambda_2(V)\]

In particular, $A_2$ is generated by the image {\bf o}f $v_1 \wedge v_2$, and $A_n= 0$ for $n\geq3$.
Then $A_\bullet$ is isomorphic to the $\F_p$-cohomology algebra of a $d$-generated Demushkin pro-$p$ group
(cf. \cite[Def.~3.9.9]{nsw:cohn}), and thus it is universally Koszul by \cite[Prop.~29]{MPPT}.
\end{exam}

\begin{exam}\label{ex:RAAG no UK}
 Let $\Gamma =(\calV,\calE)$ be a combinatorial graph without loops, and let $\Lambda_\bullet(\Gamma)$ be the exterior 
Stanley-Reisner algebra associated to $\Gamma$.
Then $\Lambda_\bullet(\Gamma)$ is Koszul (cf. \cite{priddy}, see also \cite[\S~3.2]{papa} and \cite[\S~4.2.2]{thomas}).
Moreover, by \cite[Thm.~4.6]{CQ:RAAG} the algebra $\Lambda_\bullet(\Gamma)$ is also universally Koszul if and only if $\Gamma$ has the diagonal property --- i.e., for a{\bf n}y four vertices $v_1,\ldots,v_4\in\calV$ such that
$$\{v_1 , v_2 \}, \{v_2 , v_3 \}, \{v_3 , v_4 \} \in \calE,$$
then one has $\{v_1,v_3\}\in\calE$ or $\{v_2,v_4\}\in\calE$ (see, e.g., \cite{droms}).
\end{exam}


\section{Two-relator pro-$p$ groups}\label{sec:2rel}

Henceforth, every subgroup of a pro-p group is meant to be closed with respect to the pro-p topology, and generators are topological generators.
For (closed) subgroups $H, H_1, H_2$ of a pro-$p$ group $G$ and for every $n\geq 1$, $H_n$ is the subgroup of $G$ generated by $n$-th powers of the elements of $H$, whereas $[H_1,H_2]$ is the subgroup of $G$ generated by commutators
$$[g_1,g_2]=g_1^{-1}\cdot g_1^{g_2}=g_1^{-1}g_2^{-1}g_1g_2 ,$$
with $g_1\in H_1$ and $g_2\in H_2$.


\subsection{Cohomology of pro-$p$ groups}

For a pro-$p$ group $G$ we set
\[ G_{(2)}=G^p[G,G],\qquad G_{(3)}=\begin{cases}G^p[G,[G, G]] & \text{if }p\neq2, \\ 
                                  G^4[G, G]^2[G,[G,G]]& \text{if } p= 2 \end{cases}\]
--- namely, $G_{(2)}$ and $G_{(3)}$ are the second and the third elements of the $p$-Zassenhaus filtration of $G$ (cf. \cite[\S~3.1]{MPQT}).
In particular, $G_{(2)}$ coincid{\bf e}s with the Frattini subgroup of $G$.

A short exact sequence of pro-p groups
\begin{equation}\label{eq:pres}
\xymatrix{ \{1\}\ar[r] & R\ar[r] & F\ar[r] & G\ar[r] & \{1\} } 
\end{equation}
with $F$ a free pro-$p$ group, is called a {\sl presentation} of the pro-$p$ group $G$.
If $R\subseteq F_{(2)}$, then the presentation \eqref{eq:pres} is {\sl minimal} --- roughly speaking, $F$ and $G$ have the ``same'' minimal generating system.
For a minimal presentation \eqref{eq:pres} of $G$, a set of elements of $F$ which generates minimally $R$ as normal subgroup is
called a set of {\sl defining relations} of $G$.

For a pro-$p$ group $G$ we shall denote the $\F_p$-cohomology groups $H^n(G,\F_p)$ simply by $H^n(G)$ for every $n\geq0$.
In particular, one has 
\begin{equation}\label{eq:H1}
 H^0(G)=\F_p\qquad \text{and}\qquad H^1(G)\simeq(G/G_{(2)})^\ast
\end{equation}
(cf. \cite[Prop.~3.9.1]{nsw:cohn}).
Moreover, a minimal presentation \eqref{eq:pres} of $G$ induces an exact sequence in cohomology
\begin{equation}\label{eq:5tes}
  \xymatrix{ 0\ar[r] & H^1(G)\ar[r]^-{\Inf_{F,R}^1} & H^1(F)\ar[r]^-{\Res_{F,R}^1} & H^1(R)^{F} \ar[r]^-{\mathrm{trg}_{F,R}} &
H^2(G) \ar[r]^-{\Inf_{F,R}^2} & H^2(F)}
\end{equation}
(cf. \cite[Prop.~1.6.7]{nsw:cohn}) --- if $V$ is a continuous $G$-module for a pro-$p$ group $G$, then $V^G$ denotes the subspace of $G$-invariants.
Since \eqref{eq:pres} is minimal, by \eqref{eq:H1} the map $\Inf^1_{F,R}$ is an isomorphism.
Moreover, also the map $\mathrm{trg}_{F,R}$ is an isomorphism, as $H^2(F) = 0$ (see Proposition~\ref{prop:free} below), and its inverse induces an isomorphism $\phi$ of vector spaces
\begin{equation}\label{eq:H2}
 H^2(G)\overset{\phi}{\longrightarrow}\left((R/R_{(2)})^\ast\right)^F=(R/R^p[R,F])^\ast
\end{equation}
By \eqref{eq:H1}, $\dim(H^1(G))$ is the minimal number of generators of $G$, and by \eqref{eq:H2} $\dim(H^2(G))$ is the number of defining relations of $G$.
If $H^1(G)$ and $H^2(G)$ have both finite dimension, then $G$ is said to be {\sl finitely presented}.

The $\F_p$-cohomology of a pro-$p$ group comes endowed with the bilinear {\sl cup-product}
\[
 H^i(G)\times H^j(G)\overset{\cup}{\longrightarrow} H^{i+j}(G),
\]
which is graded-commutative (cf. \cite[Ch.~I, \S~4]{nsw:cohn}).
The maximum positive integer $n$ such that $H^n(G)\neq0$ and $H^{n+1}(G)=0$ is called the cohomological dimension
of $G$, and it is denoted by $\cd(G)$ (cf. \cite[Def.~3.3.1]{nsw:cohn}).

By \eqref{eq:H2}, if $G$ is a free pro-p group then $H^2(G)= 0$.
Also the converse is true (cf. \cite[Prop.~3.5.17]{nsw:cohn}).

\begin{pro}\label{prop:free}
  A pro-$p$ group $G$ is free if and only if $\cd(G) = 1$.
\end{pro}

Let $G$ be a finitely generated pro-$p$ group with minimal presentation \eqref{eq:pres}.
We may identify $H^1(G)$ and $H^1(F)$ via the isomorphism $\Inf^1_{F,R}$. 
Also, we may identify a basis $\calX=\{x_1,\ldots,x_d\}$ of $F$ with its image in $G$. 
Let $\calB=\{a_1,\ldots,a_d\}$ be the basis of $H^1(G)$ dual to $\calX$, i.e., $a_i(x_j)=\delta_{ij}$ for $i,j\in\{1,\ldots,d\}$. For every $r\in F_{(2)}$ one may write

\begin{equation}\label{eq:relations}
\begin{split}
 r&=\prod_{i<j}[x_i,x_j]^{\alpha_{ij}}\cdot r',\qquad\text{if } p\neq2, \\ 
 r&=\prod_{i=1}^dx_i^{2\alpha_{ii}}\prod_{i<j}[x_i,x_j]^{\alpha_{ij}}\cdot r',\qquad\text{if } p=2, \\ 
\end{split}\end{equation}
for some $r'\in F_{(3)}$, with $0\leq\alpha_{ij}<p$, and these numbers are uniquely determined by $r$.
The shape of the defining relations of a pro-$p$ group and the behavior of the cup-product are related by the following (cf. \cite[Prop.~1.3.2]{vogel}).

\begin{pro}\label{prop:vogel}
 Let $G$ be a finitely presented pro-$p$ group with minimal presentation \eqref{eq:pres}, and let $\calX=\{x_1,\ldots,x_d\}$ and $\calB=\{a_1,\ldots,a_d\}$ be as above.
Given a set of defining relations $\{r_1,\ldots,r_m\}\subseteq F_{(2)}$, for every $h=1,\ldots,m$ the isomorphism $\phi$ (see \eqref{eq:H2}) induces a morphism
\[\begin{split}
   &\mathrm{tr}_h\colon H^2(G)\longrightarrow\F_p, \\  &\mathrm{tr}_h(b)= \phi(b)(r_h),
  \end{split}\]
such that, for every $1\leq i\leq j\leq d$, one has $\mathrm{tr}_h(a_i\cup a_j)=-\alpha_{ij}$, where the $\alpha_{ij}$'s are the numbers in \eqref{eq:relations} with $r = r_h$.
\end{pro}

\begin{exam}\label{exam:onerel}
 Let $G$ be a finitely generated one-relator pro-$p$ group, with minimal presentation \eqref{eq:pres} and defining relation $r$, and with $\calX=\{x_1,\ldots,x_d\}$ and $\calB=\{a_1,\ldots,a_d\}$ as above.
 Since $\dim(H^2(G))=1$ by \eqref{eq:H2}, the algebra $H^\bullet(G)$ is quadratic, and wedge-commutative if $p = 2$, if and only if $H^2(G)$ is generated by some non-trivial $a_i\cup a_j$ (and also $a_i\cup a_i=0$ for all $i=1,\ldots,d$, if $p=2$) and $H^3(G)=0$: by Proposition~\ref{prop:vogel} this occurs if and only if $\alpha_{ij}\neq0$ for some $i,j$, with the $\alpha_{ij}$'s as in \ref{eq:relations} (i.e., $r\notin F_{(3)}$), and also $\alpha_{ii}=0$ for every $i=1,\ldots,d$ if $p = 2$ (see \cite[Prop. 4.2]{cq:onerel} for the details).

 In this case, one may choose $\calX$ such that $\alpha_{1,2}=\alpha_{3,4}=\ldots=\alpha_{s-1,s}=1$, for some even $s\leq d$, and $\alpha_{ij}=0$ for all other couples $(i,j)$, so that one has an isomorphism of quadratic algebras
\[H^\bullet(G)\simeq A_\bullet\sqcup Q(V,V^{\otimes2}),\]
where $A_\bullet$ is the quadratic algebra as in Example~\ref{exam:DemushkinUK}, (with $A_1$ generated by $a_1,\ldots,a_s$, and $A_2$ generated by $a_1\cup a_2=\ldots=a_{s-1}\cup a_s$), and $V$ a finite (possibly trivial) vector space, generated by $a_i$ with $s+1\leq i\leq d$ (cf. \cite[Prop. 4.6]{cq:onerel}).
In particular, $H^\bullet(G)$ is universally Koszul by Example~\ref{exam:UK}--(b).
\end{exam}


\subsection{Two-relator pro-$p$ groups and cohomology}\label{ssec:2rel cohom}

Henceforth $G$ will be a finitely generated two-relator pro-$p$ group, with minimal presentation \eqref{eq:pres}. 
Also, the set $\calX=\{x_1,\ldots,x_d\}$ will denote a basis of $F$ (identified with its image in $G$), with $d=\dim(H^1(G))$, and $\calB=\{a_1,\ldots,a_d\}$ will be the associated dual basis of $H^1(G)$.
For simplicity, we will omit the symbol $\cup$ to denote the cup-product of two elements of $H^\bullet(G)$.
For our convenience, we slightly modify the definition given in \cite[\S~1]{BGLV}.

\begin{defn}\label{defn:quaddef}
A two-relator pro-$p$ group $G$ is {\sl quadratically defined} if the cup-product induces an epimorphism 
$H^1(G)^{\otimes2}\to H^2(G)$, and also $a\cdot a=0$ for every $a\in H^1(G)$ in the case $p=2$.
\end{defn}

By Proposition~\ref{prop:vogel}, $G$ is quadratically defined if and only if $r_1,r_2\in F_{(2)}\smallsetminus F_{(3)}$
for any set of defining relations $\{r_1,r_2\}\subseteq F_{(2)}$, and also $\alpha_{ii} =0$ for every $i=1,\ldots,d$ in the case $p=2$, where the $\alpha_{ii}$'s are the numbers in \eqref{eq:relations} with $r=r_1,r_2$ (see also \cite[Thm.~7.3]{MPQT}).

Set $I=\{1,\ldots,d\}$, and let $\succ$ denote the lexicographic order on $I^2=\{(i,j)\mid 1\leq i,j\leq d\}$ --- 
namely, $(i,j)\succ (h, k)$ if $i>h$ or if $i = h$ and $j > k$.
If $G$ is quadratically defined, by Proposition~\ref{prop:vogel} and \cite[Rem.~2.5]{QSV} one may choose defining relations 
$r_1,r_2\in F_{(2)}$ such that
\begin{equation}\label{eq:rel2}
\begin{split}
 r_1\equiv &[x_1,x_2]\cdot\prod_{\substack{1\leq i<j\leq d\\(i,j)\succ(1,2)}}[x_i,x_j]^{\alpha_{ij}}\mod F_{(3)}, \\ 
 r_2\equiv &[x_s,x_t]\cdot\prod_{\substack{1\leq i<j\leq d\\(i,j)\succ(s,t)}}[x_i,x_j]^{\beta_{ij}}\mod F_{(3)}, \\ 
\end{split}\end{equation}
for some $(s,t)\succ(1,2)$, and $0\leq\alpha_{ij},\beta_{ij}\leq p-1$, with $\alpha_{st} = 0$.
By Proposition~\ref{prop:vogel}, one has  
\[\mathrm{tr}_1(a_1a_2)=1 \qquad\text{and}\qquad \mathrm{tr}_1(a_ia_j) = \alpha_{ij} ,\]
for $i\leq j$ and $(i,j)\succ(1, 2)$, and likewise
\[\mathrm{tr}_2(a_sa_t)=1 \qquad\text{and}\qquad \mathrm{tr}_2(a_ia_j) =\begin{cases} 0 , &\text{ if }(i, j) \prec (s, t),
\\ \beta_{ij}, &\text{ if }(i, j) \succ (s, t).
 \end{cases}\]
Altogether, $\{a_1a_2,a_sa_t\}$ is a basis of $H^2(G)$, and one has relations
\begin{equation}\label{eq:relH2}
 \begin{split}
  a_i a_i &= 0, \\
a_j a_i &= -a_i a_j ,\\
a_i a_j &= \alpha_{ij}(a_1 a_2 ) + \beta_{ij} (a_s a_t ),\qquad i < j,
 \end{split}\end{equation}
were we set implicitely $\alpha_{1,2}=\beta_{st}=1$ and $\beta_{ij}=0$ for $(i, j)\prec(s,t)$.
Finally, one has the following (cf. \cite[Thm.~1--2]{BGLV}).

\begin{pro}\label{prop:cd2}
 Let $G$ be a finitely generated quadratically defined two-relator pro-$p$ group. Then $\cd(G) = 2$.
\end{pro}

As a consequence we obtain the following.

\begin{pro}\label{prop:2rel}
 Let $G$ be a finitely generated two-relator pro-p. The following are equivalent:
 \begin{itemize}
\item[(i)] $H^\bullet(G)$ is quadratic (and wedge-commutative, if $p = 2$);
\item[(ii)] $G$ is quadratically defined.
 \end{itemize}
\end{pro}

\begin{proof}
Assume first that $H^\bullet(G)$ is quadratic (and wedge-commutative, if $p = 2$).
Then one has $H^n(G)=0$ for $n\geq3$, while the cup-product induces epimorphisms
\[\begin{split}
H^1(G)^{\otimes2}\longrightarrow\Lambda_2(H^1(G))\longrightarrow H^2(G), \qquad &\text{if }p\neq2, \\
H^1(G)^{\otimes2}\longrightarrow\dfrac{S_2(H^1(G))}{\langle a^2\mid a\in H^1(G)\rangle}\longrightarrow H^2(G), 
\qquad &\text{if }p = 2,  \end{split}
\]
and thus $G$ is quadratically defined.

Assume now that $G$ is quadratically defined, and let $r_1,r_2$ be defining relations as in \eqref{eq:rel2}.
By Proposition~\ref{prop:cd2}, for $n\geq3$ one has $H^n(G)=0$, whereas $H^2(G)$ is generated by $a_1 a_2$ and $a_s a_t$,
so that $H^\bullet(G)$ is 1-generated. 
In fact, by the relations \eqref{eq:relH2} one has epimorphisms of graded algebras 
$\psi_\ast\colon \Lambda_\bullet(H^1(G))\to H^\bullet(G)$ --- if $p = 2$, with an abuse of notation we set
\[\Lambda_\bullet(H^1(G))=\frac{S_\bullet(H^1(G))}{(a^2\mid a\in H^1(G))}\]
---, with $\Ker(\psi_n)=\Lambda_n(H^1(G))$ for $n\geq3$.
We claim that $$\Ker(\psi_2)\wedge H^1(G)=\Lambda_3(H^1(G)),$$ which implies that $\Ker(\psi_\ast)$ is the two-sided ideal 
of $\Lambda_\bullet(H^1(G))$ generated by $\Ker(\psi_2)$.

By \eqref{eq:relH2}, $\Ker(\psi_2)$ is the subspace of $\Lambda_2(H^1(G))$ generated by the elements
\[b_{ij}:= a_i \wedge a_j - \alpha_{ij} (a_1 \wedge a_2 ) - \beta_{ij}(a_s \wedge a_t ),\qquad1\leq i < j \leq d.\]
First, suppose that $s=1$.
Then one has $b_{ij}=a_i\wedge a_j-a_1\wedge(\alpha_{ij}a_2+\beta_{ij}a_t)$ for every $i<j$, and thus
\begin{equation}\label{eq:proof0}
\begin{split}
  a_1 \wedge b_{2,t} &= a_1 \wedge a_2 \wedge a_t - a_1\wedge a_1\wedge(\alpha_{2,t}a_2+\beta_{2,t}a_t)\\
&= a_1 \wedge a_2 \wedge a_t - 0 , \end{split}
\end{equation}
so that $a_1 \wedge a_2 \wedge a_t\in\Ker(\psi_2)\wedge H^1(G)$.
Now, for any value of $s\geq1$ and for every $h\geq3$, one has
\begin{equation}\label{eq:proof1}
\begin{split}
  a_2 \wedge b_{1,h} &= a_2 \wedge a_1 \wedge a_h - \alpha_{1,h} \cdot a_2 \wedge (a_1 \wedge a_2 ) -\beta_{1,h}\cdot a_2\wedge (a_s \wedge a_t )\\
&= -a_1 \wedge a_2 \wedge a_h - 0 - \beta_{1,h}(a_2\wedge a_s \wedge a_t ). \end{split}
\end{equation}
If $s=1$ then by \eqref{eq:proof1} $a_1 \wedge a_2 \wedge a_h$ lies in $\Ker(\psi_2)\wedge H^1(G)$, as $a_2\wedge a_1\wedge a_t$ does.
If $s\geq2$ then $(1, h) \prec (s, t)$ --- hence $\beta_{1,h} = 0$, and \eqref{eq:proof1} yields 
$a_1 \wedge a_2 \wedge a_h\in\Ker(\psi_2)\wedge H^1(G)$.
Similarly, for every $h = 1,\ldots, d$, $h\neq s,t$, one has
\begin{equation}\label{eq:proof2}
\begin{split}
a_t\wedge b_{s,h} &= a_t\wedge a_s\wedge a_h-\alpha_{s,h}\cdot a_t\wedge(a_1\wedge a_2)-\beta_{s,h}\cdot a_t\wedge(a_s\wedge a_t)\\
&= -a_s \wedge a_t \wedge a_h - \alpha_{s,h}(a_1 \wedge a_2 \wedge a_t ) - 0.\end{split}
\end{equation}
Thus, $a_s\wedge a_t\wedge a_h\in\Ker(\psi_2)\wedge H^1(G)$, as $a_1\wedge a_2\wedge a_t\in\Ker(\psi_2)\wedge H^1(G)$ by \ref{eq:proof1}.
Finally, for every $1\leq h < k < l \leq d$ one has
\begin{equation}\label{eq:proof3}
\begin{split}
a_h\wedge b_{kl}&= a_h\wedge a_k\wedge a_l-\alpha_{kl}\cdot a_h\wedge(a_1\wedge a_2)-\beta_{kl}\cdot a_h\wedge(a_s\wedge a_t)\\
&= a_h\wedge a_k\wedge a_l-\alpha_{kl}(a_1\wedge a_2\wedge a_h)-\beta_{kl}(a_s\wedge a_t\wedge a_h),
\end{split}
\end{equation}
and thus $a_h\wedge a_k \wedge a_l \in \ker(\psi_2)\wedge H^1(G)$.
Therefore, $\Ker(\psi_2)\wedge H^1(G)=\Lambda_3(H^1(G))$, and this yields the claim.
\end{proof}


\subsection{Universally Koszul cohomology}\label{ssec:UK 2rel}

The next result shows that the $\F_p$-cohomology of a quadratically defined two-relator pro-$p$ group is universally Koszul.

\begin{thm}\label{thm:2rel}
 Let $G$ be a finitely generated quadratically defined two-relator
pro-$p$ group. Then $H^\bullet(G)$ is universally Koszul.
\end{thm}

\begin{proof}
 Set $A_\bullet= H^\bullet(G)$, and $d =\dim(A_1 )$.
By Proposition~\ref{prop:2rel}, $A_\bullet$ is quadratic and graded-commutative (wedge-commutative if $p = 2$), and $A_n = 0$
for $n \geq 3$.
Let $\calB = \{a_1 ,\ldots, a_d\}$ be a basis of $A_1$ as in \S~\ref{ssec:2rel cohom}.
Thus, $\{a_1 a_2 , a_s a_t \}$ is a basis of $A_2$, and one has the relations \eqref{eq:relH2}.

Let $I$ be an ideal of $A_\bullet$ lying in $\calL(A_\bullet )$, $I \neq A_+$, and $b \in  A_1\smallsetminus I_1$, and set
$J = I:(b)$.
Since $A_n = 0$ for $n \geq 3$, one has $A_2 \cdot (b) = 0 \subseteq I$, and thus $J_2 = A_2$.
In order to show that $J \in  \calL(A_\bullet )$, we need to show that $A_2$ is generated by
$J_1$, i.e., $J_1 \cdot A_1 = A_2$.
Since $b^2 = 0$, one has $b \in  J_1$.
One has three cases: $\dim(b \cdot A_1 ) = 0, 1, 2$.

Suppose first that $\dim(b \cdot A_1 ) = 0$, i.e., $ba = 0$ for every $a \in  A_1$.
Then $b\cdot A_1=0 \subseteq I_2$, and hence $A_1 \subseteq J_1$. Therefore, $J_1 \cdot A_1 = A_1 \cdot A_1 = A_2$.

Suppose now that $\dim(b \cdot A_1 ) = 1$, i.e., $b \cdot A_1$ is generated by $\alpha a_1 a_2 + \beta a_s a_t$,
for some $\alpha , \beta  \in  \F_p$, with $\alpha$, $\beta$ not both 0.
Then
\[ a_1 b = \lambda_1 (\alpha a_1 a_2 + \beta a_s a_t ),\qquad
a_2 b = \lambda_2 (\alpha a_1 a_2 + \beta a_s a_t )\]
for some $\lambda_1 , \lambda_2 \in\F_p$.
If $\lambda_1 = 0$ then $a_1 b = 0 \in  I_2$, and $a_1 \in  J_1$; similarly, if $\lambda_2 = 0$ then $a_2 \in  J_1$.
If $\lambda_1,\lambda_2\neq 0$, then $(a_1-\lambda_1/\lambda_2 a_2)b=0\in I_2$, and $(a_1-\lambda_1/\lambda_2 a_2)\in J_1$.
In both cases, $a_1 a_2 \in  J_1 \cdot A_1$.
Likewise,
\[a_s b = \mu_1 (\alpha a_1 a_2 + \beta a_s a_t ),\qquad
a_t b = \mu_2 (\alpha a_1 a_2 + \beta a_s a_t )\]
for some $\mu_1,\mu_2\in\F_p$.
If $\mu_1=0$ then $a_sb=0\in I_2$, and $a_s\in J_1$; similarly, if $\mu_2=0$ then $a_t\in J_1$.
If $\mu_1,\mu_2\neq0$, then $(a_s-\mu_1/\mu_2a_t)b=0\in I_2$, and $(a_s-\mu_1/\mu_2a_t)\in J_1$.
In both cases, $a_s a_t \in J_1 \cdot A_1$.
Therefore, $A_2 \subseteq J_1 \cdot A_1$, and this concludes the case $\dim(b \cdot A_1 ) = 1$.

Finally, suppose that $\dim(b \cdot A_1 )= 2$, i.e., $b\cdot A_1 = A_2$.
Since $b \in  J_1$, one has $A_2 = b \cdot A_1 \subseteq J_1 \cdot A_1$. This concludes the proof.
\end{proof}

Now we may prove Theorem~\ref{thm:intro}.

\begin{proof}[Proof of Theorem~1.2]
Let $G$ be a finitely generated pro-$p$ group with at most two defining relations, and suppose that $H^\bullet(G)$ is quadratic, and moreover $a^2 = 0$ for every $a \in H^1(G)$ if $p = 2$.
If $G$ has no relations, then it is a free pro-$p$ group, and by Proposition~\ref{prop:free} $H^\bullet(G)\simeq Q(V,V^{\otimes2})$ for some finite vector space $V$.
Hence, Example \ref{exam:UK}--(a) yields the claim.
If $G$ is one-relator, then by \cite[Prop.~4.2]{cq:onerel} $G$ is as in Example~\ref{exam:onerel}, and thus $H^\bullet(G)$ is universally Koszul.
Finally, if $G$ is two-relator, we apply Theorem~\ref{thm:2rel}.
\end{proof}

The next example shows that one may not extend Theorem~\ref{thm:intro} to finitely
generated pro-$p$ groups with quadratic $\F_p$-cohomology, cohomological dimension equal to 2 and more than two defining relations.

\begin{exam}\label{exam:square}\rm
 Let $G$ be a pro-$p$ group with minimal presentation
 \begin{equation}\label{eq:RAAG}
 G = \left\langle x, y, z, t \mid [x, y]x^{\lambda_1}y^{\mu_1} = [y, z]y^{\lambda_2}z^{\mu_2} = [z, t]z^{\lambda_3}t^{\mu_3}= 1\right\rangle,
 \end{equation}
with $\lambda_i,\mu_i\in p\Z_p$ for $i=1,2,3$ (and moreover $\lambda_i,\mu_i\in 4\Z_2$ if $p=2$).
Then $G$ is a {\sl generalised right-angled Artin pro-$p$ group} (cf. \cite[\S~5.1]{QSV}) associated to the
graph $\Gamma$ with vertices $\calV =\{x, y, z, t\}$ and edges $\calE=\{\{x, y\},\{y,z\},\{z,t\}\}$ --- i.e.,
$\Gamma$ is a path of length 3, and $\Gamma$ does not have the diagonal property.
In particular, if all $\lambda_i$'s and $\mu_i$'s are 0, then $G$ is the pro-$p$ completion of the (discrete) right-angled Artin group associated to $\Gamma$.

The cohomology algebra $H^\bullet(G)$ of $G$ is isomorphic to the exterior Stanley-Reisner algebra $\Lambda_\bullet(\Gamma)$ (cf. \cite[Thm. F]{QSV}).
In particular, $H^\bullet(G)$ is Koszul and $H^3(G) = 0$, but $H^\bullet(G)$ is not universally Koszul by Example~\ref{ex:RAAG no UK}.
\end{exam}


\section{Maximal pro-$p$ Galois groups}\label{sec:Galois}

For a field $\K$, let $\calG_{\K}$ denote the maximal pro-$p$ Galois group of $\K$.
If $\K$ contains a root of 1 of order $p$, then by Kummer theory one has an isomorphism
\begin{equation}\label{eq:Kummer}
 \K^\times/(\K^\times)^p\simeq H^1(\calG_{\K}),
\end{equation}
where $\K^\times=\K\smallsetminus\{0\}$ denotes the multiplicative group of $\K$.
Moreover, let $\mathrm{Br}_p(\K)$
denote the $p$-part of the Brauer group $\mathrm{Br}(\K)$ of $\K$ --- i.e, $\mathrm{Br}_p(\K)$ is the subgroup
of $\mathrm{Br}(\K)$ generated by all elements of order $p$.
Then one has an isomorphism
\begin{equation}\label{eq:brauer}
\mathrm{Br}_p(\K)\simeq H^2(\calG_{\K} )
\end{equation}
(cf. \cite[Thm.~6.3.4]{nsw:cohn}).

The {\sl mod-$p$ Milnor $K$-ring} of $\K$ is the graded $\F_p$-algebra 
\[
 K_\bullet^M(\K)_{/p}=\bigoplus_{n\geq0}K_n^M(\K)_{/p}=\frac{T_\bullet(\K^\times)}{(\Omega)}\otimes_{\Z}\F_p 
\]
where $T_\bullet(\K^\times)$ is the tensor algebra over $\Z$ generated by $\K^\times$, and $(\Omega)$ the two-sided ideal of $T_\bullet(\K^\times)$ generated by the elements $\alpha\otimes(1-\alpha)$ with $\alpha\in\K^\times\smallsetminus\{1\}$ (see, e.g., \cite[Def.~6.4.1]{nsw:cohn}).
Thus $K_\bullet^M(\K)_{/p}$ is quadratic, with $K_1^M(\K)_{/p}=\K^\times/(\K^\times)^p$.
If $\K$ contains a root of 1 of order $p$, by the Rost-Voevodsky Theorem the isomorphism \eqref{eq:Kummer} induces an isomorphism of $\F_p$-algebras 
\begin{equation}\label{eq:RV}
 K_\bullet^M(\K)_{/p}\overset{\sim}{\longrightarrow} H^\bullet(\calG_{\K})
\end{equation}
(cf. \cite{voev,weibel}, see also \cite[\S~24.3]{efrat:book}), and thus $H^\bullet(\calG_\K)$ is quadratic.
The following is a well-know fact on $\F_2$-cohomology of maximal pro-$2$ Galois groups of fields.

\begin{lem}\label{lemma:k2}
 Let $\K$ be a field containing $\sqrt{-1}$. Then the $\F_2$-cohomology algebra $H^\bullet(\calG_{\K})$ of the maximal pro-2 Galois group of $\K$ is wedge-commutative.
\end{lem}

\begin{proof}
By \eqref{eq:RV}, it sufficies to show that $K_\bullet^M(\K)_{/2}$ is wedge-commutative.
For $\alpha,\beta\in\K^\times$, set $\{\alpha\}=\alpha(\K^\times)^2\in \K^\times/(\K^\times)^2$ and let $\{\alpha,\beta\}$ denote the image of $\alpha\otimes\beta$ in $K_2^M(\K)_{/2}$.
By definition, for every $\alpha\in\K^\times$ one has $\{\alpha\}\cdot \{\alpha\}=\{\alpha,-1\}\in K_2^M(\K)_{/2}$.
Hence, if $\sqrt{-1}\in\K$ then $-1\in(\K^\times)^2$, i.e., $\{-1\}$ is trivial in $K_1^M(\K)_{/2}=\K^\times/(\K^\times)^2$, and hence $\{\alpha\}\cdot\{\alpha\}=\{\alpha,-1\}$ is trivial in $K_2^M(\K)_{/2}$.
\end{proof}

From Theorem~\ref{thm:intro} one deduces the following.

\begin{cor}\label{cor:Galois}
 Let $\K$ be a field containing a root of 1 of order p (and also $\sqrt{-1}$ if $p = 2$), and suppose that the quotient 
 $\K^\times/(\K^\times)^p$ is finite. If $\rk(\mathrm{Br}_p(\K))\leq2$, then $H^\bullet(\calG_\K)$ is universally Koszul.
\end{cor}

\begin{proof}
By \eqref{eq:Kummer} and \eqref{eq:H1}, the pro-$p$ group $\calG_{\K}$ is finitely generated.
Moreover, by \eqref{eq:brauer} and \eqref{eq:H2} $\calG_K$ has at most two defining relations.
Also, $H^\bullet(\calG_K)$ is quadratic by the Rost-Voevodsky Theorem.

If $\mathrm{Br}_p(\K)$ is trivial, then $H^n(\calG_{\K})=0$ for all $n\geq2$, namely, $\cd(\calG_{\K})= 1$,
and $\calG_{\K}$ is a free pro-$p$ group by Proposition~\ref{prop:free}.
If $\rk(\mathrm{Br}_p(\K))=1$, then $\calG_{\K}$ is one-relator, and if $p = 2$ then $H^\bullet(\calG_{\K})$ is wedge-commutative by Lemma~\ref{lemma:k2}.

Finally, if $\rk(\mathrm{Br}_p(\K))=2$, then $\calG_{\K}$ is a two-relator pro-$p$ group.
If $p\neq2$, then $\calG_{\K}$ is quadratically defined by Proposition~\ref{prop:2rel}.
If $p=2$, then $H^\bullet(\calG_{\K})$ is wedge-commutative by Lemma~\ref{lemma:k2}, and by Proposition~\ref{prop:2rel} $\calG_{\K}$ is quadratically defined.
Altogether, Theorem~\ref{thm:intro} yields the claim.
\end{proof}

The isomorphisms \eqref{eq:H2} and \eqref{eq:brauer} imply that Corollary~\ref{cor:Galois} is equivalent to Corollary~\ref{cor:intro}.

\begin{rem}\label{rem:PZ}\rm
 Let $G$ be the pro-$p$ group with minimal presentation \eqref{eq:RAAG}, with $\lambda_i=\mu_i=0$ for $i=1,2,3$. 
It was recently shown by I.~Snopce and P.~Zalesskii that $G$ does not occur as maximal pro-$p$ Galois group $\calG_{\K}$ for any field $\K$ containing a root of 1 of order $p$ (cf. \cite{SZ}).
Therefore, Example~\ref{exam:square} for this case does not provide a counterexample to Conjecture~\ref{conjecture:intro}.

On the other hand, if $G$ is a pro-$p$ group with minimal presentation \eqref{eq:RAAG} where not all $\lambda_i$'s and $\mu_i$'s are 0, it is not known (but for some special cases, see \cite[\S~5.7]{QSV}) whether $G$ may occur as maximal pro-$p$ Galois group $\calG_{\K}$ for some field $\K$ containing a root of 1 of order $p$.
We conjecture it does not for any choice of the $\lambda_i$'s and $\mu_i$'s; and therefore we expect Example~\ref{exam:square} not to provide counterexamples to Conjecture~\ref{conjecture:intro}.
\end{rem}

{\small \subsection*{acknowledgements}
The author wishes to thank the anonymous referee for her/his useful comments.
This paper was inspired by a talk (a {\sl mathematical salad}) that F.W. Pasini delivered in Milan, Italy, in Dec. 2018, and also by the discussions with him and with J. Min\'a\v{c} and N.D. T\^an on Koszul algebras and Galois cohomology --- therefore, the author is grateful to all these people. Also, the author wishes to remark
that J.P. Labute's work on Demushkin groups, mild pro-$p$ groups and filtrations of
groups influenced deeply the author's way of thinking about Koszul property in Galois cohomology
}

\end{document}